\documentclass[letterpaper,12pt]{article}
\usepackage[margin=1in]{geometry}
\usepackage{verbatim}
\usepackage{amsmath}
\usepackage{amssymb}
\usepackage{graphicx}
\usepackage{amsthm}

\def\epsilon{\varepsilon}
\def\bigoh{\mathcal{O}}

\def\mod{\,\mathop{\rm mod}\,}
\def\9dots{\vdots\,\vdots\,\vdots}
\def\phi{\varphi}
\def\th{{\rm th}}
\def\ith{{\it th}}

\def\0s{{\bf 0}}

\newtheorem{theorem}{Theorem}[section]
\newtheorem{lemma}[theorem]{Lemma}

\newtheorem{observe}[theorem]{Observation}
\newtheorem{remark1}[theorem]{Remark}

\newenvironment{observation}{\begin{observe} \rm}{\end{observe}}
\newenvironment{remark}{\begin{remark1} \rm}{\end{remark1}}

\title{A fast randomized algorithm for orthogonal projection}
\author{Vladimir Rokhlin and Mark Tygert}
\date{December 10, 2009}

\begin{document}

\maketitle

\begin{abstract}
We describe an algorithm that, given any full-rank matrix $A$
having fewer rows than columns,
can rapidly compute the orthogonal projection of any vector
onto the null space of $A$, as well as the orthogonal projection
onto the row space of $A$,
provided that both $A$ and its adjoint $A^*$ can be applied rapidly
to arbitrary vectors.
As an intermediate step, the algorithm solves
the overdetermined linear least-squares regression involving $A^*$
(and so can be used for this, too).
The basis of the algorithm is an obvious but numerically unstable scheme;
suitable use of a preconditioner yields numerical stability.
We generate the preconditioner rapidly via a randomized procedure
that succeeds with extremely high probability.
In many circumstances, the method can accelerate interior-point methods
for convex optimization, such as linear programming
(Ming Gu, personal communication).
\end{abstract}

\section{Introduction}

This article introduces an algorithm that,
given any full-rank matrix $A$ having fewer rows than columns, such that
both $A$ and its adjoint $A^*$ can be applied rapidly to arbitrary vectors,
can rapidly compute
\begin{enumerate}
\item the orthogonal projection of any vector $b$ onto the null space of $A$,
\item the orthogonal projection of $b$ onto the row space of $A$, or
\item the vector $x$ minimizing the Euclidean norm $\| A^* \, x - b \|$
of the difference between $A^* \, x$ and the given vector $b$
(thus solving the overdetermined linear least-squares regression
$A^* \, x \approx b$).
\end{enumerate}
For simplicity,\,we focus on projecting
onto the null space,\,describing extensions in Remark\,\ref{least-squares}.

The basis for our algorithm is the well-known formula
for the orthogonal projection $Zb$ of a vector $b$ onto the null space
of a full-rank matrix $A$ having fewer rows than columns:
\begin{equation}
\label{basic_identity}
Zb = b - A^* \, (A \, A^*)^{-1} \, A \, b.
\end{equation}
It is trivial to verify that $Z$ is the orthogonal projection
onto the null space of $A$, by checking
that $ZZ = Z$ (so that $Z$ is a projection),
that $Z^* = Z$ (so that $Z$ is self-adjoint,
making the projection ``orthogonal''), and
that $AZ = 0$ (so that $Zb$ is in the null space of $A$ for any vector $b$).
Since we are assuming that $A$ and $A^*$ can be applied rapidly
to arbitrary vectors, (\ref{basic_identity}) immediately yields
a fast algorithm for computing $Zb$.
Indeed, we may apply $A$ to each column of $A^*$ rapidly,
obtaining the smaller square matrix $A \, A^*$;
we can then apply $A$ to $b$, solve $(A \, A^*) \, x = A \, b$ for $x$,
apply $A^*$ to $x$, and subtract the result from $b$, obtaining $Zb$.
However, using~(\ref{basic_identity}) directly is numerically unstable;
using~(\ref{basic_identity}) directly is analogous
to using the normal equations for computing the solution
to a linear least-squares regression (see, for example,~\cite{golub-van_loan}).

We replace~(\ref{basic_identity}) with a formula that is equivalent
in exact arithmetic (but not in floating-point), namely
\begin{equation}
\label{basic_precond_intro}
Zb = b
   - A^* \, \left(P^*)^{-1} \, (P^{-1} \, A \, A^* \, (P^*)^{-1}\right)^{-1}
     \, P^{-1} \, A \, b,
\end{equation}
where $P$ is a matrix such that $P^{-1} \, A$ is well-conditioned.
If $A$ is very short and fat, then $P$ is small and square,
and so we can apply $P^{-1}$ and $(P^*)^{-1}$ reasonably fast
to arbitrary vectors.
Thus, given a matrix $P$ such that $P^{-1} \, A$ is well-conditioned,
(\ref{basic_precond_intro}) yields a numerically stable fast algorithm
for computing $Zb$.
We can rapidly produce such a matrix $P$ via the randomized method
of~\cite{rokhlin-tygert}, which is based on the techniques
of~\cite{drineas-mahoney-muthukrishnan-sarlos} and its predecessors.
Although the method of the present paper is randomized,
the probability of numerical instability is negligible
(see Remark~\ref{stability} in Section~\ref{algorithm} below).
A preconditioned iterative approach similar to that in~\cite{rokhlin-tygert}
is also feasible.
Related work includes~\cite{avron-maymounkov-toledo}
and~\cite{gotsman-toledo}.

The remaining sections cover the following:
Section~\ref{prelims} summarizes relevant facts from prior publications.
Section~\ref{apparatus} details the mathematical basis for the algorithm.
Section~\ref{algorithm} describes the algorithm of the present paper
in detail, and estimates its computational costs.
Section~\ref{numerical} reports the results of several numerical experiments.

In the present paper, the entries of all matrices are real-valued,
though the algorithm and analysis extend trivially to matrices
whose entries are complex-valued.
We use the term ``condition number'' to refer to the $l^2$ condition number,
the greatest singular value divided by the least
(the least is the $m^\th$ greatest singular value,
for an $m \times n$ matrix with $m \le n$).

\section{Preliminaries}
\label{prelims}

In this section, we summarize several facts about the spectra
of random matrices and about techniques for preconditioning.

\subsection{Singular values of random matrices}
\label{random_singular_values}

The following lemma provides a highly probable upper bound
on the greatest singular value
of a matrix whose entries are independent, identically distributed
(i.i.d.) Gaussian random variables of zero mean and unit variance;
Formula~8.8 in~\cite{goldstine-von_neumann} provides an equivalent formulation
of the lemma.

\begin{lemma}
\label{greatest_value}
Suppose that $l$ and $m$ are positive integers with $l \ge m$.
Suppose further that $H$ is an $m \times l$ matrix whose entries are
i.i.d.\ Gaussian random variables of zero mean and unit variance,
and $\alpha$ is a positive real number, such that
$\alpha > 1$ and
\begin{equation}
\label{failure_prob}
\pi_+ = 1 - \frac{1}{4 \, (\alpha^2-1) \, \sqrt{\pi l \alpha^2}}
        \left( \frac{2 \alpha^2}{e^{\alpha^2-1}} \right)^l
\end{equation}
is nonnegative.

Then, the greatest singular value of $H$ is no greater than
$\sqrt{2l} \, \alpha$ with probability no less than $\pi_+$ defined
in~(\ref{failure_prob}).
\end{lemma}

The following lemma provides a highly probable lower bound
on the least singular value
of a matrix whose entries are i.i.d.\ Gaussian
random variables of zero mean and unit variance;
Formula~2.5 in~\cite{chen-dongarra}
and the proof of Lemma~4.1 in~\cite{chen-dongarra}
together provide an equivalent formulation of Lemma~\ref{least_value}.

\begin{lemma}
\label{least_value}
Suppose that $l$ and $m$ are positive integers with $l \ge m$.
Suppose further that $H$ is an $m \times l$ matrix whose entries are
i.i.d.\ Gaussian random variables of zero mean and unit variance,
and $\beta$ is a positive real number, such that
\begin{equation}
\label{failure_prob2}
\pi_- = 1 - \frac{1}{\sqrt{2 \pi \, (l-m+1)}}
     \, \left( \frac{e}{(l-m+1) \, \beta} \right)^{l-m+1}
\end{equation}
is nonnegative.

Then, the least (that is, the $m^\ith$ greatest) singular value
of $H$ is no less than $1 / (\sqrt{l} \; \beta)$
with probability no less than $\pi_-$ defined in~(\ref{failure_prob2}).
\end{lemma}

The following lemma provides a highly probable upper bound
on the condition number
of a matrix whose entries are i.i.d.\ Gaussian
random variables of zero mean and unit variance;
the lemma follows immediately from Lemmas~\ref{greatest_value}
and~\ref{least_value}.
For simpler bounds, see~\cite{chen-dongarra}.

\begin{lemma}
\label{rand_lemma}
Suppose that $l$ and $m$ are positive integers with $l \ge m$.
Suppose further that $H$ is an $m \times l$ matrix whose entries are
i.i.d.\ Gaussian random variables of zero mean and unit variance,
and $\alpha$ and $\beta$ are positive real numbers, such that
$\alpha > 1$ and
\begin{equation}
\label{failure_prob3}
\pi_0 = 1 - \frac{1}{4 \, (\alpha^2-1) \, \sqrt{\pi l \alpha^2}}
       \left( \frac{2 \alpha^2}{e^{\alpha^2-1}} \right)^l
     - \frac{1}{\sqrt{2 \pi \, (l-m+1)}}
    \, \left( \frac{e}{(l-m+1) \, \beta} \right)^{l-m+1}
\end{equation}
is nonnegative.

Then, the condition number of $H$ is no greater than
$\sqrt{2} \, l \alpha \beta$
with probability no less than $\pi_0$ defined in~(\ref{failure_prob3}).
\end{lemma}

\subsection{Preconditioning}
\label{preconditioning}

The following lemma, proven in a slightly different form
as Theorem~1 in~\cite{rokhlin-tygert}, states that,
given a short, fat $m \times n$ matrix $A$, the condition number
of a certain preconditioned version of $A$ corresponding
to an $n \times l$ matrix $G$ is equal to the condition number
of $V^* \, G$, where $V$ is an $n \times m$ matrix of orthonormal
right singular vectors of $A$.

\begin{lemma}
\label{cond_lemma}
Suppose that $l$, $m$, and $n$ are positive integers such that $m \le l \le n$.
Suppose further that $A$ is a full-rank $m \times n$ matrix,
and that the SVD of $A$ is
\begin{equation}
\label{original_SVD}
A_{m \times n} = U_{m \times m} \, \Sigma_{m \times m} \, (V_{n \times m})^*,
\end{equation}
where the columns of $U$ are orthonormal, the columns of $V$ are orthonormal,
and $\Sigma$ is a diagonal matrix whose entries are all nonnegative.
Suppose in addition that $G$ is an $n \times l$ matrix such that
the $m \times l$ matrix $A \, G$ has full rank.

Then, there exist an $m \times m$ matrix $P$,
and an $l \times m$ matrix $Q$ whose columns are orthonormal, such that
\begin{equation}
\label{QR_decomposition}
A_{m \times n} \, G_{n \times l} = P_{m \times m} \, (Q_{l \times m})^*.
\end{equation}

Furthermore, the condition numbers of $P^{-1} \, A$ and $V^* \, G$ are equal,
for any $m \times m$ matrix~$P$, and $l \times m$ matrix $Q$
whose columns are orthonormal,
such that $P$ and $Q$ satisfy~(\ref{QR_decomposition}).
\end{lemma}

\section{Mathematical apparatus}
\label{apparatus}

In this section, we describe a randomized method for preconditioning
rectangular matrices and prove that it succeeds with very high probability.

The following theorem states that a certain preconditioned version
$P^{-1} \, A$ of a short, fat matrix $A$ is well-conditioned
with very high probability (see Observation~\ref{high_prob}
for explicit numerical bounds).

\begin{theorem}
\label{the_theorem}
Suppose that $l$, $m$, and $n$ are positive integers such that $m \le l \le n$.
Suppose further that $A$ is a full-rank $m \times n$ matrix,
and that the SVD of $A$ is
\begin{equation}
\label{original_SVD2}
A_{m \times n} = U_{m \times m} \, \Sigma_{m \times m} \, (V_{n \times m})^*,
\end{equation}
where the columns of $U$ are orthonormal, the columns of $V$ are orthonormal,
and $\Sigma$ is a diagonal matrix whose entries are all nonnegative.
Suppose in addition that $G$ is an $n \times l$ matrix 
whose entries are i.i.d.\ Gaussian random variables
of zero mean and unit variance.
Suppose finally that $\alpha$ and $\beta$ are positive real numbers, such that
$\alpha > 1$ and $\pi_0$ defined in~(\ref{failure_prob3}) is nonnegative.

Then, there exist an $m \times m$ matrix $P$,
and an $l \times m$ matrix $Q$ whose columns are orthonormal, such that
\begin{equation}
\label{QR_decomposition2}
A_{m \times n} \, G_{n \times l} = P_{m \times m} \, (Q_{l \times m})^*.
\end{equation}

Furthermore, the condition number of $P^{-1} \, A$
is no greater than $\sqrt{2} \, l \alpha \beta$ with probability
no less than $\pi_0$ defined in~(\ref{failure_prob3}),
for any $m \times m$ matrix $P$, and $l \times m$ matrix $Q$
whose columns are orthonormal,
such that $P$ and $Q$ satisfy~(\ref{QR_decomposition2}).
\end{theorem}

\begin{proof}
Combining the facts that the columns of $V$ are orthonormal
and that the entries of $G$ are i.i.d.\ Gaussian random variables
of zero mean and unit variance yields that the entries
of $H_{m \times l} = (V_{n \times m})^* \, G_{n \times l}$
are also i.i.d.\ Gaussian random variables of zero mean and unit variance.
Combining this latter fact and Lemmas~\ref{rand_lemma} and~\ref{cond_lemma}
completes the proof.
\end{proof}

\begin{observation}
\label{high_prob}
To elucidate the bounds given in Theorem~\ref{the_theorem},
we look at some special cases.
With $m \ge 2$ and $\alpha \ge 2$, we find that
$\pi_0$ from~Theorem~\ref{the_theorem}, defined in~(\ref{failure_prob3}),
satisfies
\begin{multline}
\label{simple_bound}
\pi_0 \ge 1 - \frac{1}{4 \, (\alpha^2-1) \, \sqrt{\pi (l-m+2) \alpha^2}}
              \left( \frac{2 \alpha^2}{e^{\alpha^2-1}} \right)^{l-m+2} \\
            - \frac{1}{\sqrt{2 \pi \, (l-m+1)}}
           \, \left( \frac{e}{(l-m+1) \, \beta} \right)^{l-m+1}.
\end{multline}
Combining~(\ref{simple_bound}) and Theorem~\ref{the_theorem} yields
strong bounds on the condition number of $P^{-1} \, A$ when $l-m = 4$
and $m \ge 2$:
With $l-m = 4$, $\alpha^2 = 4$, and $\beta = 3$,
the condition number is at most $10 l$ with probability at least $1-10^{-4}$.
With $l-m = 4$, $\alpha^2 = 7$, and $\beta = 26$,
the condition number is at most $100 l$ with probability at least $1-10^{-9}$.
With $l-m = 4$, $\alpha^2 = 9$, and $\beta = 250$,
the condition number is at most $1100l$ with probability at least $1-10^{-14}$.
Moreover, if instead of $l-m = 4$ we take $l$ to be a few times $m$,
then the condition number of $P^{-1} \, A$ is no greater than a small constant
(that does not depend on $l$), with very high probability
(see~\cite{chen-dongarra}).
\end{observation}

\section{Description of the algorithm}
\label{algorithm}

In this section, we describe the algorithm of the present paper in some detail.
We tabulate its computational costs in Subsection~\ref{costs}.

Suppose that $m$ and $n$ are positive integers with $m \le n$,
and $A$ is a full-rank $m \times n$ matrix.
The orthogonal projection of a vector $b_{n \times 1}$
onto the row space of $A_{m \times n}$ is $A^* \, (A \, A^*)^{-1} \, A \, b$.
Therefore, the orthogonal projection of $b_{n \times 1}$ onto the null space
of $A_{m \times n}$ is $b - A^* \, (A \, A^*)^{-1} \, A \, b$.
After constructing a matrix $P_{m \times m}$ such that the condition number
of $P^{-1} \, A$ is not too large,
we may compute the orthogonal projection onto the null space of $A$
via the identity
\begin{equation}
\label{basic_precond}
b - A^* \, (A \, A^*)^{-1} \, A \, b
= b - A^* \, \left(P^*)^{-1} \, (P^{-1} \, A \, A^* \, (P^*)^{-1}\right)^{-1}
   \, P^{-1} \, A \, b.
\end{equation}
With such a matrix $P_{m \times m}$,
(\ref{basic_precond}) provides a numerically stable means
for computing the orthogonal projection of $b$ onto the null space of $A$.

To construct a matrix $P_{m \times m}$ such that the condition number
of $P^{-1} \, A$ is not too large, we choose a positive integer $l$
such that $m \le l \le n$ ($l = m+4$ is often a good choice)
and perform the following two steps:
\begin{enumerate}
\item Construct the $m \times l$ matrix
\begin{equation}
\label{testing}
S_{m \times l} = A_{m \times n} \, G_{n \times l}
\end{equation}
one column at a time, where $G$ is a matrix whose entries
are i.i.d. random variables. Specifically, generate each column of $G$
in succession, applying $A$ to the column (and saving the result)
before generating the next column.
(Notice that we do not need to store all $nl$ entries of $G$ simultaneously.)
\item Construct a pivoted $QR$ decomposition of $S^*$,
\begin{equation}
\label{QR_decomp}
(S_{m \times l})^* = Q_{l \times m} \, R_{m \times m} \, \Pi_{m \times m},
\end{equation}
where the columns of $Q$ are orthonormal, $R$ is upper-triangular,
and $\Pi$ is a permutation matrix.
(See, for example, Chapter~5 of~\cite{golub-van_loan}
for details on computing the pivoted $QR$ decomposition.)
\end{enumerate}
With $P = \Pi^* \, R^*$,
Theorem~\ref{the_theorem} and Observation~\ref{high_prob} show that
the condition number of $P^{-1} \, A$ is not too large,
with very high probability.

To construct the matrix
\begin{equation}
\label{square_inv}
Y = \left(P^{-1} \, A \, A^* \, (P^*)^{-1}\right)^{-1}
\end{equation}
which appears in the right-hand side of~(\ref{basic_precond}),
we perform the following three steps:
\begin{enumerate}
\item Construct the $m \times m$ matrix
\begin{equation}
\label{scary}
P^-_{m \times m} = \Pi_{m \times m}^* \, R_{m \times m}^{-1}
\end{equation}
(recalling that $\Pi$ is a permutation matrix).
\item For each $k = 1$,~$2$, \dots, $m-1$,~$m$,
construct the $k^\th$ column $x^{(k)}$ of an $m \times m$ matrix $X$
via the following five steps:
\begin{enumerate}
\item[a.] Extract the $k^\th$ column $c^{(k)}_{m \times 1}$ of $P^-$.
\item[b.] Construct
$d^{(k)}_{n \times 1} = (A_{m \times n})^* \, c^{(k)}_{m \times 1}$.
\item[c.] Construct
$e^{(k)}_{m \times 1} = A_{m \times n} \, d^{(k)}_{n \times 1}$.
\item[d.] Construct
$f^{(k)}_{m \times 1} = \Pi_{m \times m} \, e^{(k)}_{m \times 1}$
(recalling that $\Pi$ is a permutation matrix).
\item[e.] Solve
$R_{m \times m}^* \, x^{(k)}_{m \times 1} = f^{(k)}_{m \times 1}$
for $x^{(k)}_{m \times 1}$ (recalling that $R$ is an upper-triangular matrix).
\end{enumerate}
Notice that $X = P^{-1} \, A \, A^* \, (P^*)^{-1}$,
where $P = (R \, \Pi)^*$ and $(P^*)^{-1} = P^-$.
Also, because we generate $X$ one column at a time,
we do not need to store an extra $\bigoh(nm)$ floating-point words of memory
at any stage of the algorithm.
\item Construct the $m \times m$ matrix
\begin{equation}
\label{inversion}
Y_{m \times m} = X_{m \times m}^{-1}.
\end{equation}
\end{enumerate}

It follows from~(\ref{basic_precond}) that the orthogonal projection of $b$
onto the null space of $A$ is
\begin{equation}
b - A^* \, (A \, A^*)^{-1} \, A \, b
= b - A^* \, (P^*)^{-1} \, Y \, P^{-1} \, A \, b.
\end{equation}

In order to apply $A^* \, (P^*)^{-1} \, Y \, P^{-1} \, A$
to a vector $b_{n \times 1}$, where $P = \Pi^* \, R^*$,
we perform the following seven steps:

\begin{enumerate}
\item Construct $c_{m \times 1} = A_{m \times n} \, b_{n \times 1}$.
\item Construct $d_{m \times 1} = \Pi_{m \times m} \, c_{m \times 1}$
(recalling that $\Pi$ is a permutation matrix).
\item Solve $R_{m \times m}^* \, e_{m \times 1} = d_{m \times 1}$
for $e_{m \times 1}$ (recalling that $R$ is an upper-triangular matrix).
\item Construct $f_{m \times 1} = Y_{m \times m} \, e_{m \times 1}$.
\item Solve $R_{m \times m} \, g_{m \times 1} = f_{m \times 1}$
for $g_{m \times 1}$ (recalling that $R$ is an upper-triangular matrix).
\item Construct $h_{m \times 1} = \Pi^*_{m \times m} \, g_{m \times 1}$
(recalling that $\Pi$ is a permutation matrix).
\item Construct
$\tilde{b}_{n \times 1} = (A_{m \times n})^* \, h_{m \times 1}$.
\end{enumerate}

The output $\tilde{b} = A^* \, (P^*)^{-1} \, Y \, P^{-1} \, A \, b$
of the above procedure is the orthogonal projection of $b$
onto the row space of $A$. The orthogonal projection of $b$ onto the null space
of $A$ is then $b-\tilde{b}$.

\begin{remark}
\label{least-squares}
The vector $h_{m \times 1}$ constructed in Step~6 above
minimizes the Euclidean norm
$\| (A_{m \times n})^* \, h_{m \times 1} - b_{n \times 1} \|$
of $(A_{m \times n})^* \, h_{m \times 1} - b_{n \times 1}$; that is,
$h_{m \times 1}$ solves the overdetermined linear least-squares regression
$(A_{m \times n})^* \, h_{m \times 1} \approx b_{n \times 1}$.
\end{remark}

\begin{remark}
\label{stability}
The algorithm of the present section is numerically stable provided
that the condition number of $P^{-1} \, A$ is not too large.
Theorem~\ref{the_theorem} and Observation~\ref{high_prob} show that
the condition number is not too large, with very high probability.
\end{remark}

\subsection{Costs}
\label{costs}

In this subsection, we estimate the number of floating-point operations
required by the algorithm of the present section.

We denote by $C_A$ the cost of applying $A_{m \times n}$
to an $n \times 1$ vector;
we denote by $C_{A^*}$ the cost of applying $(A_{m \times n})^*$
to an $m \times 1$ vector.
We will be keeping in mind our assumption that $m \le l \le n$.

Constructing a matrix $P_{m \times m}$ such that the condition number
of $P^{-1} \, A$ is not too large incurs the following costs:
\begin{enumerate}
\item Constructing $S$ defined in~(\ref{testing}) costs
$l \cdot (C_A + \bigoh(n))$.
\item Constructing $R$ and $\Pi$ satisfying~(\ref{QR_decomp})
costs $\bigoh(l \cdot m^2)$.
\end{enumerate}

Given $R$ and $\Pi$ (whose product is $P^*$),
constructing the matrix $Y_{m \times m}$ defined in~(\ref{square_inv})
incurs the following costs:
\begin{enumerate}
\item Constructing $P^-$ defined in~(\ref{scary}) costs $\bigoh(m^3)$.
\item Constructing $X$ via the five-step procedure (Steps~a--e) costs
$m \cdot (C_{A^*} + C_A + \bigoh(m^2))$.
\item Constructing $Y$ in~(\ref{inversion}) costs $\bigoh(m^3)$.
\end{enumerate}

Summing up, we see from $m \le l \le n$ that constructing $R$, $\Pi$, and $Y$
defined in~(\ref{square_inv}) costs
\begin{equation}
C_{\rm pre} = (l+m) \cdot C_A + m \cdot C_{A^*} + \bigoh(l \cdot (m^2+n))
\end{equation}
floating-point operations overall,
where $C_A$ is the cost of applying $A_{m \times n}$
to an $n \times 1$ vector,
and $C_{A^*}$ is the cost of applying $(A_{m \times n})^*$
to an $m \times 1$ vector ($l = m+4$ is often a good choice).
Given $R$ and $\Pi$ (whose product is $P^*$) and $Y$,
computing the orthogonal projection $b-\tilde{b}$ of a vector $b$
onto the null space of $A$ via the seven-step procedure above costs
\begin{equation}
C_{\rm pro} = C_A + \bigoh(m^2) + C_{A^*}
\end{equation}
floating-point operations,
where again $C_A$ is the cost of applying $A_{m \times n}$
to an $n \times 1$ vector,
and $C_{A^*}$ is the cost of applying $(A_{m \times n})^*$
to an $m \times 1$ vector.

\begin{remark}
We may use iterative refinement to improve the accuracy of $h$
discussed in Remark~\ref{least-squares},
as in Section~5.7.3 of~\cite{dahlquist-bjorck}.
Each additional iteration of improvement costs $C_A + \bigoh(m^2) + C_{A^*}$
(since $R$, $\Pi$, and $Y$ are already available).
We have also found that reprojecting the computed projection
onto the null space often increases the accuracy
to which $A$ annihilates the projection.
However, this extra accuracy is not necessarily meaningful;
see Remark~\ref{perturbation} below.
\end{remark}

\section{Numerical results}
\label{numerical}

In this section, we illustrate the performance of the algorithm
of the present paper via several numerical examples.

Tables~\ref{big_time}--\ref{big_err_norm} display the results
of applying the algorithm to project onto the null space
of the sparse matrix $A_{m \times n}$ defined as follows.
First, we define a circulant matrix $B_{m \times m}$ via the formula
\begin{equation}
\label{pert_lap}
B_{j,k} = \left\{ \begin{array}{rl}
                    1, & j = k-2 \mod m \\
                   -4, & j = k-1 \mod m \\
                  6+d, & j = k          \\
                   -4, & j = k+1 \mod m \\
                    1, & j = k+2 \mod m \\
                    0, & \hbox{otherwise}
                  \end{array} \right.
\end{equation}
for $j,k = 1$,~$2$, \dots, $m-1$,~$m$,
where $d$ is a real number that we will set shortly.
Taking $n$ to be a positive integer multiple of $m$,
we define the matrix $A_{m \times n}$ via the formula
\begin{equation}
\label{test_mat}
A_{m \times n} = U_{m \times m}
\cdot \left( \begin{array}{c|c|c|c|c} B_{m \times m} & B_{m \times m} &
\cdots & B_{m \times m} & B_{m \times m} \end{array} \right)_{m \times n}
\cdot V_{n \times n},
\end{equation}
where $U_{m \times m}$ and $V_{n \times n}$ are drawn uniformly
at random from the sets of $m \times m$ and $n \times n$ permutation matrices,
and $B_{m \times m}$ is defined in~(\ref{pert_lap}).
The condition number of $B$ is easily calculated to be $\kappa = (16+d)/d$;
therefore, the condition number of $A$ is also $\kappa = (16+d)/d$.
In accordance with this, we chose $d = 16/(\kappa-1)$.

For Tables~\ref{big_time} and~\ref{big_err},
we took the entries of $G$ in~(\ref{testing}) to be
(Mitchell-Moore-Brent-Knuth) lagged Fibonacci pseudorandom sequences,
uniformly distributed on $[-1,1]$,
constructed using solely floating-point additions and subtractions
(with no integer arithmetic);
see, for example, Section~7.1.5 of~\cite{press-teukolsky-vetterling-flannery}.
For Tables~\ref{big_time_norm} and~\ref{big_err_norm},
we took the entries of $G$ in~(\ref{testing}) to be
high-quality pseudorandom numbers drawn from a Gaussian distribution
of zero mean and unit variance.
Though we used Gaussian variates in previous sections
in order to simplify the theoretical analysis,
there appears to be no practical advantage to using high-quality
truly Gaussian pseudorandom numbers;
as reflected in the tables, the very quickly generated
lagged Fibonacci sequences perform just as well in our algorithm.

Tables~\ref{low_time} and~\ref{low_err} display the results
of applying the algorithm to project onto the null space
of the dense matrix $\tilde{A}_{m \times n}$ defined via the formula
\begin{equation}
\label{modified}
\tilde{A}_{m \times n}
= A_{m \times n} + E_{m \times 10} \cdot F_{10 \times n} / \sqrt{mn},
\end{equation}
where $A_{m \times n}$ is the sparse matrix defined in~(\ref{test_mat}),
and the entries of $E_{m \times 10}$ and $F_{10 \times n}$
are i.i.d. Gaussian random variables of zero mean and unit variance.
As for Tables~\ref{big_time}--\ref{big_err_norm},
we chose $d = 16/(\kappa-1)$ for use in~(\ref{pert_lap});
with this choice, $\kappa$ is the condition number of $A$ in~(\ref{modified}),
and provides a rough estimate of the condition number of $\tilde{A}$.
We took advantage of the special structure of $\tilde{A}$
(as the sum of a sparse matrix and a rank-10 matrix)
in order to accelerate the application of $\tilde{A}$ and its adjoint
to vectors.
For Tables~\ref{low_time} and~\ref{low_err}
(just as in Tables~\ref{big_time} and~\ref{big_err}),
we took the entries of $G$ in~(\ref{testing}) to be
(Mitchell-Moore-Brent-Knuth) lagged Fibonacci pseudorandom sequences,
uniformly distributed on $[-1,1]$,
constructed using solely floating-point additions and subtractions
(with no integer arithmetic);
see, for example, Section~7.1.5 of~\cite{press-teukolsky-vetterling-flannery}.

\begin{itemize}
\item[] The following list describes the headings of the tables:
\item $m$ is the number of rows in the matrix for which we are computing
the orthogonal projection onto the null space. 
\item $n$ is the number of columns in the matrix for which we are computing
the orthogonal projection onto the null space.
\item $l$ is the number of columns in the random matrix $G$
from~(\ref{testing}).
\item $\kappa$ is the condition number of the matrix $A$
defined in~(\ref{test_mat}).
\item $s_{\rm pre}$ is the time in seconds required to compute $A \, A^*$
(or $\tilde{A} \, \tilde{A}^*$, for Table~\ref{low_time})
and its pivoted $QR$ decomposition.
\item $s_{\rm pro}$ is the time in seconds required
to compute directly the orthogonal projection
$b - A^* \, (A \, A^*)^{-1} \, A \, b$
(or $b - \tilde{A}^* \, (\tilde{A} \, \tilde{A}^*)^{-1} \, \tilde{A} \, b$,
for Table~\ref{low_time}) for a vector $b$,
having already computed a pivoted $QR$ decomposition
of $A \, A^*$ (or $\tilde{A} \, \tilde{A}^*$, for Table~\ref{low_time}).
Thus, $s_{\rm pre} + j \cdot s_{\rm pro}$ is the time required
by a standard method for orthogonally projecting $j$ vectors
onto the null space.
\item $t_{\rm pre}$ is the time in seconds required to construct the matrices
$R$, $\Pi$, and $Y$ defined in Section~\ref{algorithm}.
\item $t_{\rm pro}$ is the time in seconds required by the randomized algorithm
to compute the orthogonal projection of a vector onto the null space,
having already computed the matrices $R$, $\Pi$, and $Y$
defined in Section~\ref{algorithm}.
Thus, $t_{\rm pre} + j \cdot t_{\rm pro}$ is the time required by the algorithm
of the present paper for orthogonally projecting $j$ vectors
onto the null space.
\item $\delta_{\rm norm}$ is a measure of the error
of a standard method for orthogonally projecting onto the null space.
Specifically, $\delta_{\rm norm}$ is the Euclidean norm of the result
of applying $A$ (or $\tilde{A}$, for Table~\ref{low_err})
to the orthogonal projection (onto the null space) of a random unit vector $b$,
computed as $b - A^* \, (A \, A^*)^{-1} \, A \, b$
or $b - \tilde{A}^* \, (\tilde{A} \, \tilde{A}^*)^{-1} \, \tilde{A} \, b$.
(In Tables~\ref{big_err},~\ref{big_err_norm}, and~\ref{low_err},
$\delta_{\rm norm}$ is the maximum encountered over 100 independent
realizations of the random vector $b$.)
The tables report $\delta_{\rm norm}$ divided by the condition number,
since this is generally more informative (see Remark~\ref{perturbation} below).
\item $\epsilon_{\rm norm}$ is a measure of the error of a standard method
for orthogonally projecting onto the null space.
Specifically, $\epsilon_{\rm norm}$ is the Euclidean norm of the difference
between the orthogonal projection (onto the null space)
of a random unit vector $b$
(computed as $b - A^* \, (A \, A^*)^{-1} \, A \, b$
or $b - \tilde{A}^* \, (\tilde{A} \, \tilde{A}^*)^{-1} \, \tilde{A} \, b$)
and its own orthogonal projection
(computed as $\tilde{b} - A^* \, (A \, A^*)^{-1} \, A \, \tilde{b}$ or
$\tilde{b} - \tilde{A}^* \, (\tilde{A} \, \tilde{A}^*)^{-1} \, \tilde{A}
\, \tilde{b}$), where $\tilde{b} = b - A^* \, (A \, A^*)^{-1} \, A \, b$ or
$\tilde{b}
= b - \tilde{A}^* \, (\tilde{A} \, \tilde{A}^*)^{-1} \, \tilde{A} \, b$
(as computed).
(In Tables~\ref{big_err},~\ref{big_err_norm}, and~\ref{low_err},
$\epsilon_{\rm norm}$ is the maximum encountered over 100
independent realizations of the random vector $b$.)
The tables report $\epsilon_{\rm norm}$ divided by the condition number,
since this is generally more informative (see Remark~\ref{perturbation} below).
\item $\delta_{\rm rand}$ is a measure of the error
of the algorithm of the present paper.
Specifically, $\delta_{\rm rand}$ is the Euclidean norm of the result
of applying $A$ (or $\tilde{A}$, for Table~\ref{low_err})
to the orthogonal projection (onto the null space) of a random unit vector $b$
(with the projection computed via the randomized algorithm).
(In Tables~\ref{big_err},~\ref{big_err_norm}, and~\ref{low_err},
$\delta_{\rm rand}$ is the maximum encountered over 100 independent
realizations of the random vector $b$.)
The tables report $\delta_{\rm rand}$ divided by the condition number,
since this is generally more informative (see Remark~\ref{perturbation} below).
\item $\epsilon_{\rm rand}$ is a measure of the error
of the algorithm of the present paper.
Specifically, $\epsilon_{\rm rand}$ is the Euclidean norm of the difference
between the orthogonal projection (onto the null space)
of a random unit vector $b$ (with the projection computed
via the randomized algorithm),
and the projection of the projection of $b$
(also computed via the algorithm of the present paper).
(In Tables~\ref{big_err},~\ref{big_err_norm}, and~\ref{low_err},
$\epsilon_{\rm rand}$ is the maximum encountered over 100
independent realizations of the random vector $b$.)
The tables report $\epsilon_{\rm rand}$ divided by the condition number,
since this is generally more informative (see Remark~\ref{perturbation} below).
\end{itemize}

\begin{table}
\caption{Timings for the sparse matrix $A$ defined in~(\ref{test_mat})}
\label{big_time}
\vspace{1em}
\begin{tabular*}{\columnwidth}{@{\extracolsep{\fill}}rrrlllll}
$m$ & $n$ & $l$ & $\kappa$ & $s_{\rm pre}$ & $s_{\rm pro}$ & $t_{\rm pre}$ & $t_{\rm pro}$ \\\hline\hline
  40 & 1000000 &   44 &  1E8 & .23E2 & .91E0   & .35E2 & .90E0   \\
 400 & 1000000 &  404 &  1E8 & .23E3 & .90E0   & .35E3 & .90E0   \\
4000 & 1000000 & 4004 &  1E8 & .25E4 & .10E1   & .38E4 & .12E1   \\\hline
1000 &    3000 & 1004 &  1E8 & .29E1 & .85E--2 & .19E2 & .15E--1 \\
1000 &   30000 & 1004 &  1E8 & .69E1 & .16E--1 & .25E2 & .22E--1 \\
1000 &  300000 & 1004 &  1E8 & .13E3 & .21E0   & .20E3 & .21E0   \\
1000 & 3000000 & 1004 &  1E8 & .19E4 & .30E1   & .28E4 & .30E1   \\\hline
2000 & 1000000 & 2000 &  1E8 & .12E4 & .92E0   & .19E4 & .96E0   \\
2000 & 1000000 & 4000 &  1E8 & .12E4 & .93E0   & .25E4 & .98E0   \\
2000 & 1000000 & 8000 &  1E8 & .12E4 & .95E0   & .36E4 & .98E0   \\\hline
4000 & 1000000 & 4004 &  1E4 & .26E4 & .11E1   & .39E4 & .12E1   \\
4000 & 1000000 & 4004 &  1E6 & .25E4 & .10E1   & .39E4 & .12E1   \\
4000 & 1000000 & 4004 &  1E8 & .25E4 & .10E1   & .38E4 & .12E1   \\
4000 & 1000000 & 4004 & 1E10 & .25E4 & .11E1   & .38E4 & .12E1   \\
4000 & 1000000 & 4004 & 1E12 & .25E4 & .10E1   & .39E4 & .12E1
\end{tabular*}
\end{table}

\begin{table}
\caption{Errors for the sparse matrix $A$ defined in~(\ref{test_mat})}
\label{big_err}
\vspace{1em}
\begin{tabular*}{\columnwidth}{@{\extracolsep{\fill}}rrrlllll}
$m$ & $n$ & $l$ & $\kappa$ & $\delta_{\rm norm}/\kappa$ & $\epsilon_{\rm norm}/\kappa$ & $\delta_{\rm rand}/\kappa$ & $\epsilon_{\rm rand}/\kappa$ \\\hline\hline
  40 & 1000000 &   44 &  1E8 & .20E--14 & .21E--13 & .20E--15 & .61E--14 \\
 400 & 1000000 &  404 &  1E8 & .23E--14 & .67E--11 & .48E--15 & .11E--14 \\
4000 & 1000000 & 4004 &  1E8 & .25E--14 & .16E--08 & .17E--14 & .24E--15 \\\hline
1000 &    3000 & 1004 &  1E8 & .17E--14 & .20E--08 & .85E--15 & .12E--15 \\
1000 &   30000 & 1004 &  1E8 & .27E--13 & .54E--06 & .11E--14 & .95E--16 \\
1000 &  300000 & 1004 &  1E8 & .15E--14 & .16E--09 & .76E--15 & .35E--15 \\
1000 & 3000000 & 1004 &  1E8 & .22E--14 & .57E--11 & .11E--14 & .96E--15 \\\hline
2000 & 1000000 & 2000 &  1E8 & .23E--14 & .32E--09 & .12E--14 & .33E--15 \\
2000 & 1000000 & 4000 &  1E8 & .23E--14 & .32E--09 & .80E--15 & .17E--15 \\
2000 & 1000000 & 8000 &  1E8 & .23E--14 & .32E--09 & .75E--15 & .17E--15 \\\hline
4000 & 1000000 & 4004 &  1E4 & .22E--13 & .79E--12 & .59E--14 & .69E--15 \\
4000 & 1000000 & 4004 &  1E6 & .46E--14 & .15E--10 & .32E--14 & .38E--15 \\
4000 & 1000000 & 4004 &  1E8 & .25E--14 & .16E--08 & .17E--14 & .24E--15 \\
4000 & 1000000 & 4004 & 1E10 & .37E--17 & .13E--12 & .89E--15 & .15E--15 \\
4000 & 1000000 & 4004 & 1E12 & .59E--19 & .13E--14 & .45E--15 & .13E--15
\end{tabular*}
\end{table}

\begin{table}
\caption{Timings using normal variates}
\label{big_time_norm}
\vspace{1em}
\begin{tabular*}{\columnwidth}{@{\extracolsep{\fill}}rrrlllll}
$m$ & $n$ & $l$ & $\kappa$ & $s_{\rm pre}$ & $s_{\rm pro}$ & $t_{\rm pre}$ & $t_{\rm pro}$ \\\hline\hline
4000 & 1000000 & 4004 &  1E4 & .26E4 & .11E1   & .54E4 & .12E1   \\
4000 & 1000000 & 4004 &  1E6 & .25E4 & .10E1   & .54E4 & .12E1   \\
4000 & 1000000 & 4004 &  1E8 & .25E4 & .10E1   & .54E4 & .12E1   \\
4000 & 1000000 & 4004 & 1E10 & .25E4 & .11E1   & .56E4 & .13E1   \\
4000 & 1000000 & 4004 & 1E12 & .25E4 & .10E1   & .54E4 & .12E1
\end{tabular*}
\end{table}

\begin{table}
\caption{Errors using normal variates}
\label{big_err_norm}
\vspace{1em}
\begin{tabular*}{\columnwidth}{@{\extracolsep{\fill}}rrrlllll}
$m$ & $n$ & $l$ & $\kappa$ & $\delta_{\rm norm}/\kappa$ & $\epsilon_{\rm norm}/\kappa$ & $\delta_{\rm rand}/\kappa$ & $\epsilon_{\rm rand}/\kappa$ \\\hline\hline
4000 & 1000000 & 4004 &  1E4 & .22E--13 & .79E--12 & .59E--14 & .68E--15 \\
4000 & 1000000 & 4004 &  1E6 & .46E--14 & .15E--10 & .36E--14 & .38E--15 \\
4000 & 1000000 & 4004 &  1E8 & .25E--14 & .16E--08 & .20E--14 & .23E--15 \\
4000 & 1000000 & 4004 & 1E10 & .37E--17 & .13E--12 & .95E--15 & .15E--15 \\
4000 & 1000000 & 4004 & 1E12 & .59E--19 & .13E--14 & .43E--15 & .14E--15
\end{tabular*}
\end{table}

\begin{table}
\caption{Timings for the dense matrix $\tilde{A}$ defined in~(\ref{modified})}
\label{low_time}
\vspace{1em}
\begin{tabular*}{\columnwidth}{@{\extracolsep{\fill}}rrrlllll}
$m$ & $n$ & $l$ & $\kappa$ & $s_{\rm pre}$ & $s_{\rm pro}$ & $t_{\rm pre}$ & $t_{\rm pro}$ \\\hline\hline
  40 & 1000000 &   44 &  1E8 & .28E2 & .11E1   & .42E2 & .11E1   \\
 400 & 1000000 &  404 &  1E8 & .28E3 & .11E1   & .41E3 & .11E1   \\
4000 & 1000000 & 4004 &  1E8 & .30E4 & .12E1   & .56E4 & .15E1   \\\hline
1000 &    3000 & 1004 &  1E8 & .31E1 & .89E--2 & .19E2 & .15E--1 \\
1000 &   30000 & 1004 &  1E8 & .89E1 & .19E--1 & .28E2 & .25E--1 \\
1000 &  300000 & 1004 &  1E8 & .16E3 & .25E0   & .25E3 & .26E0   \\
1000 & 3000000 & 1004 &  1E8 & .23E4 & .34E1   & .33E4 & .34E1
\end{tabular*}
\end{table}

\begin{table}
\caption{Errors for the dense matrix $\tilde{A}$ defined in~(\ref{modified})}
\label{low_err}
\vspace{1em}
\begin{tabular*}{\columnwidth}{@{\extracolsep{\fill}}rrrlllll}
$m$ & $n$ & $l$ & $\kappa$ & $\delta_{\rm norm}/\kappa$ & $\epsilon_{\rm norm}/\kappa$ & $\delta_{\rm rand}/\kappa$ & $\epsilon_{\rm rand}/\kappa$ \\\hline\hline
  40 & 1000000 &   44 &  1E8 & .19E--16 & .48E--16 & .34E--18 & .15E--19 \\
 400 & 1000000 &  404 &  1E8 & .42E--16 & .32E--14 & .68E--17 & .31E--18 \\
4000 & 1000000 & 4004 &  1E8 & .72E--14 & .11E--08 & .14E--14 & .27E--16 \\\hline
1000 &    3000 & 1004 &  1E8 & .48E--15 & .12E--10 & .28E--15 & .52E--16 \\
1000 &   30000 & 1004 &  1E8 & .38E--15 & .12E--10 & .25E--15 & .17E--16 \\
1000 &  300000 & 1004 &  1E8 & .65E--15 & .17E--11 & .30E--15 & .13E--16 \\
1000 & 3000000 & 1004 &  1E8 & .14E--14 & .22E--11 & .28E--15 & .57E--16
\end{tabular*}
\end{table}

\begin{remark}
\label{perturbation}
Standard perturbation theory shows that,
if we add to entries of the matrix~$A$ random numbers of size $\delta$,
and the unit vector $b$ has a substantial projection onto the null space
of $A$, then the entries of the vector $h$ minimizing
the Euclidean norm $\| A^* \, h - b \|$ change by about $\delta \cdot C^2$,
where $C$ is the condition number of $A$
(see, for example, formula 1.4.26 in~\cite{bjorck1}).
Thus, it is generally meaningless to compute $h$ more accurately
than $\epsilon \cdot C^2$, where $\epsilon$ is the machine precision,
and $C$ is the condition number.
Similarly, it is generally meaningless to compute the orthogonal projection
onto the null space of $A$ more accurately than $\epsilon \cdot C$,
where $\epsilon$ is the machine precision,
and $C$ is the condition number of $A$.
\end{remark}

We performed all computations using IEEE standard double-precision variables,
whose mantissas have approximately one bit of precision less than 16 digits
(so that the relative precision of the variables is approximately .2E--15).
We ran all computations on one core
of a 1.86~GHz Intel Centrino Core Duo microprocessor
with 2~MB of L2 cache and 1~GB of RAM.
We compiled the Fortran~77 code
using the Lahey/Fujitsu Linux Express v6.2 compiler,
with the optimization flag {\tt {-}{-}o2} enabled.
We used plane (Householder) reflections to compute all pivoted $QR$
decompositions (see, for example, Chapter~5 of~\cite{golub-van_loan}).

\begin{remark}
The numerical results reported here and our further experiments
indicate that the timings of the classical scheme and the randomized method
are similar, whereas the randomized method produces more accurate projections
(specifically, the randomized projections are idempotent to higher precision).
The quality of the pseudorandom numbers does not affect the accuracy
of the algorithm much, if at all, nor does replacing the normal variates
with uniform variates for the entries of $G$ in~(\ref{testing}).
\end{remark}

\section*{Acknowledgements}

We would like to thank Professor Ming Gu of UC Berkeley for his insights.
We are grateful for the grants supporting this research,
ONR grant N00014-07-1-0711, AFOSR grant FA9550-09-1-0241,
and DARPA/AFOSR grant FA9550-07-1-0541.


\bibliographystyle{siam}
\bibliography{spre}

\end{document}